\documentclass[12pt]{amsart}
\usepackage{amsmath}
\usepackage{amssymb}
\usepackage{amsfonts}
\numberwithin{equation}{section}
\newcommand{\D}{\mathbb{D}\,}

\DeclareMathOperator{\grad}{grad}

\newtheorem{theorem}{Theorem}[section]
\newtheorem{lemma}[theorem]{Lemma}
\newtheorem{proposition}[theorem]{Proposition}

\theoremstyle{definition}

\theoremstyle{remark}
\newtheorem{remark}[theorem]{Remark}

\numberwithin{equation}{section}

\renewcommand{\Re}{{\rm Re}\,}
\renewcommand{\Im}{{\rm Im}\,}

\begin{document}

\title[Boundary uniqueness and spectral subspaces]
{Boundary uniqueness of harmonic functions and
spectral subspaces of operator groups}
%    Information for third author
\author{Alexander Borichev}
\address{Centre de Math\'ematiques et
Informatique, Universit\'e d'Aix-Mar\-seille I, 39 rue Fr\'ed\'eric
Joliot-Curie, 13453 Marseille, France}
 \email{borichev@cmi.univ-mrs.fr}
%\thanks{The third author was partially supported by a KBN grant.}
\author{Yuri Tomilov}
\address{Department of Mathematics and Computer Science, Nicolas Copernicus University, ul. Chopina 12/18, 87-100 Torun,
Poland and Institute of Mathematics, Polish Academy of Sciences,
\' Sniadeckich str. 8, 00-956 Warsaw, Poland}
\email{tomilov@mat.uni.torun.pl}
\thanks{This was completed with the support of the Marie Curie Transfer of Knowledge program, project "TODEQ".
The first author was also partially supported the ANR projects
DYNOP and FRAB. The second author was partially supported the MNiSzW grant
 Nr. N201384834.}

\begin{abstract}
We obtain new uniqueness theorems for harmonic fun\-ctions
defined on the unit disc or in the half plane. These results are applied
to obtain new resolvent descriptions of spectral subspaces of
polynomially bounded groups of operators on Banach spaces.
\end{abstract}

%\date{\today}

\maketitle

\section{Introduction}

It is a simple fact that a function $u$ harmonic on the unit disc
and having zero limits at the boundary is equal to $0$. If however
we replace (unrestricted) limits at the boundary by restricted
limits (say, radial ones or non tangential ones), then the
statement becomes false, see e.g. Proposition \ref{proposition}
below. One way to get the uniqueness in this case is to
impose certain growth conditions on $u$.

To get an intuition on what kind of restrictions on $u$ can be
imposed we recall two known uniqueness theorems for harmonic
functions in the unit disc. Given a function $u$ continuous on the
unit disc $\D$, we set
$$
M_r(u)=\max_{0\le\theta\le 2\pi}u(re^{i\theta}),\qquad 0<r<1.
$$
\begin{theorem} {\rm(V.~L.~Shapiro, \cite{SHA})} If $u$ is harmonic in $\D$
and
\begin{gather*}
M_r(|u|)=o((1-r)^{-2}),\qquad r\to 1-,\\
\lim_{z=re^{i\varphi}\to e^{i\varphi}}u(z)=0,\qquad  \varphi\in [0,2\pi],
\end{gather*}
then $u=0$.
\label{t1}
\end{theorem}
\begin{theorem} {\rm(F.~Wolf, \cite[Section 7]{WOL})} Let $k>0$. If $u$ is harmonic in $\D$
and
\begin{gather*}
\log^+M_r(|u|)=o((1-r)^{-\pi/(2\arctan (1/k))}),\qquad r\to 1-,\\
\lim_{z=re^{i\theta}\to
e^{i\varphi},\,|\varphi-\theta|<k(1-r)}u(z)=0, \qquad  \varphi\in [0,2\pi],
\end{gather*}
then $u=0$.
\label{t2}
\end{theorem}
Both results are the best possible as far as the growth assumptions on
$u$ are concerned, see e.g. \cite{BoChTo07}. On the other hand, it
is natural to ask, for which approach restrictions, the polynomial
limit growth of $|u|$ with respect to $(1-r)^{-1}$ in the above
results can be replaced by the exponential one; or more generally,
what is the relation between the size of the approach domain and
the limit growth of $|u|$ in boundary uniqueness theorems for
harmonic $u$. In a different situation, a relation between
the (Dirichlet spaces) smoothness and the boundary limits along tangential approach domains
was observed in \cite{NaRuSh82}.

Note that  as a consequence of the results in \cite{NIK, BOR}, in the situations we consider here, it
is sufficient to impose growth restrictions just on $M_r(u)$. Next
we define a scale of approach domains. We say that a
non-decreasing continuous function $h:[0,1]\to[0,1]$, $h(0)=0$, is
an approach function. Given an approach function $h$ consider
$\Delta^h=\{x+iy:|x|\le h(y),\, 0<y<1\}\subset \mathbb C_+
=\{x+iy\in\mathbb C:y>0\}$. Given $\varphi\in[0,2\pi]$, the
function $f_\varphi:z\mapsto e^{i\varphi}(i-z)/(i+z)$ maps
$\mathbb C_+$ onto the unit disc. Set
$\Omega^h(\varphi)=f_\varphi(\Delta^h)$. The following two
theorems give a partial answer to the above question in the case
of polynomial type approach domains.

\begin{theorem}\label{mmm}
Let $u$ be harmonic in the unit disc $\mathbb D$,
\begin{equation}
\lim_{z\in \Omega^h(\varphi),\,z\to e^{i\varphi}}u(z)=0, \qquad
%\text{for every} \,\,
\varphi\in [0,2\pi]. \label{lim}
\end{equation}

{\rm (a) (V.~L.~Shapiro)} If $h(t)=0$ and
$$
M_r(u)=o((1-r)^{-2}),\qquad r\to 1-,
$$
then $u=0$.

{\rm (b)} If $h(t)=ct^3$, $c>0$, and
$$
\log^+M_r(u)=o((1-r)^{-1}),\qquad r\to 1-,
$$
then $u=0$.

{\rm (c) (F.~Wolf)} If $h(t)=ct$, $c>0$,  and
$$
\log^+M_r(u)=o((1-r)^{-\pi/(2\arctan (1/c))}),\qquad r\to 1-,
$$
then $u=0$.

{\rm (d)} If $h(t)=t^\gamma$, $0<\gamma<1$,  and
$$
\log^+\log^+M_r(u)=o((1-r)^{\gamma-1}),\qquad r\to 1-,
$$
then $u=0$.
\label{te1}
\end{theorem}

Of course, the result in (c) remains valid if $h(t)=ct+c_1t^2$
for small $t$ and for some $c_1\in\mathbb R$.

The statements of Theorem~\ref{te1} are sharp with respect to the
polynomial scale of approach domains.

\begin{proposition}\label{proposition}

{\rm (a)} If $h(t)=o(t^3)$, $t\to 0$, then there exists
a function $u\not=0$ harmonic in $\mathbb D$ and satisfying \eqref{lim}
such that
$$
M_r(|u|)=O((1-r)^{-2}),\qquad r\to 1-.
$$

{\rm (b)} If $1<\gamma\le 3$, $h(t)=t^\gamma$, $\varepsilon>0$,
then there exists a function $u\not=0$ harmonic in $\mathbb D$ and satisfying \eqref{lim} such that
$$
\log^+M_r(|u|)\le\varepsilon (1-r)^{-1}, \qquad r \in (0,1).
$$

{\rm (c)} If $h(t)=ct$, $c>0$, $\varepsilon>0$,
then there exists a function $u\not=0$ harmonic in $\mathbb D$ and satisfying \eqref{lim} such that
$$
\log^+M_r(|u|)\le\varepsilon(1-r)^{-\pi/(2\arctan (1/c))}, \qquad
r \in (0,1).
$$

{\rm (d)} If $0<\gamma<1$, $h(t)=t^\gamma$,
then for some $k=k(\gamma)$ there exists a function $u\not=0$ harmonic in
$\mathbb D$ and satisfying \eqref{lim} such that
$$
\log^+\log^+M_r(|u|)\le k(1-r)^{\gamma-1},\qquad r \in (0,1).
$$
\label{pp}
\end{proposition}

Next we pass to an application of our
function-theoretical results (in their half-plane version) to the
study of spectral properties of operator groups on Banach spaces.

 Let $X$ be a Banach space, and let
$(T(t))_{t \in \mathbb R}$ be a $C_0$-group on $X$ with generator
$A$ growing at most polynomially, i.e. such that $\|T(t)\|\le M (1
+ |t|^a), a \ge 0.$ There is an extensive literature on spectral
properties of polynomially growing (and, even more generally, non
quasianalytic) $C_0$-groups on Banach spaces, see e.g.
\cite{AmMoGe96}, \cite{Ba79}, \cite{BaWo83},
\cite{Jo82}--\cite{Kr2},
%\cite{Jo92}, \cite{Kanto}, \cite{Ko1},
%\cite{Ko2}, \cite{Kr1}, \cite{Kr2},
 \cite{Ma81}, \cite{Ma86}.
 However, the problem of description of spectral maximal
subspaces of the groups (or equivalently, of the generators) in
terms of nontangential boundary behavior of local resolvents of
their generators has not been addressed so far (see, though,
\cite{BaChTo02,Jo92}). On the other hand, such characterizations look
very natural being interpreted as the study of regularity of
operator-valued distributions (measures) on the real line by means
of their Poisson integrals.

Denote by $\sigma(x)$ the local spectrum of $x$ (i.e. the
complement of the set of points $\lambda$ such that the local
resolvent $R(\lambda,A)x=(\lambda-A)^{-1}x$ is analytic in a
neighborhood of $\lambda$). For general information on the local
spectra see \cite{LaNe00}, \cite{Mu07}. Let
$$
X(F):=\{x\in X : \sigma (x) \subset  iF \}
$$
be the spectral (maximal) subspace of $A$ (equivalently of $(T(t))_{t \in \mathbb R}$, see, for instance,  \cite{Jo82})
corresponding to the closed subset $F$ of $\mathbb R,$ and let
$R(\lambda, A):=(\lambda-A)^{-1}$ be the resolvent of $A$ defined
at least on $\mathbb C\setminus i\mathbb R.$
% Let ${\mathcal S}(\mathbb R)$
%be the space of Schwartz test functions.
 Recall   that if $f$ belongs to the  space of Schwartz test functions ${\mathcal
S} (\mathbb R)$, then
 %for $f$ from  ${\mathcal S} (\mathbb R),$
%one has
\begin{equation*}\label{pois}
 \int_{-\infty}^{\infty} e^{-\alpha |t|} \hat f(t) T(t)\, dt
=\int_{-\infty}^{\infty}  f(\beta)\left(R(\alpha+i\beta,
A)-R(-\alpha+i\beta, A) \right) \, d\beta,
\end{equation*}
where $\hat f(t):= \int_{\mathbb R} e^{-ist} f(s)\, ds$ and the
integrals converge in the strong sense.
 Hence the operator-valued distribution
\begin{eqnarray*}
E &:& {\mathcal S} (\mathbb R)\mapsto \mathcal L(X),\\
 \langle E
,f \rangle &:=& \lim_{\alpha \to 0+ } \int_{\mathbb R}
f(\beta)\left( R(\alpha+i\beta, A)-R(-\alpha+i\beta, A) \right) \,
d\beta\\&=& \int_{\mathbb R} \hat f(t) T(t) \, dt
\end{eqnarray*}
is well-defined.

Let $x$ be such that  $\sigma(x)$ is compact, and
$E_x:=E(\cdot)x.$ Then, denoting
$$D(\alpha + i\beta)=:R(\alpha+i\beta, A)-R(-\alpha+i\beta, A),
\qquad \alpha \neq 0,$$ and following the observation in \cite[p.
139]{Jo92}, we obtain that
\begin{equation}\label{poisson}
D(\alpha +i\beta)x = 2\pi\bigl \langle E_x,
\frac{1}{\pi}\frac{\alpha}{\alpha^2 + (\beta - \cdot)^2} \bigr
\rangle = 2\pi\bigl( E_x * P_{\alpha}\bigr)(\beta),
\end{equation}
where \eqref{poisson} is understood as the distributional
convolution of the distribution having compact support with  the
$C^{\infty}(\mathbb R)$-function.  (Remark  also  that
$$
D(\alpha +i\beta)= \widehat {e^{-\alpha t} T(\cdot)} (\beta),
\qquad \alpha >0.)
$$
Thus $D(\cdot)x$ can be treated as the Poisson integral of $E_x.$
(Using distributions in ${\mathcal D}_{L_p}'$ one can write down
the Poisson integral of $E_x$ without any restrictions of
$\sigma(x),$ see e.g. \cite{AlGuPe07}. This however would
unnecessarily complicate the presentation.)

If $X$ is a Hilbert space with inner product $(\cdot, \cdot)$ and
$(T(t))_{t \ge 0}$ is a unitary $C_0$-group in $X$ so that $-iA$
is self adjoint, then the distribution $E$ can be identified with
a multiple of the spectral measure  $E(\cdot)$ of  $-iA$, and the
relation \eqref{poisson} can be interpreted using the usual
convolution. In particular,
\begin{equation}\label{poisson1}
\bigl( D(\alpha +i\beta)x ,x \bigr) = 2 \int_{-\infty}^{\infty}
\frac{\alpha \,  d \left( E(t)x, x \right)}{\alpha^2 + (\beta -
t)^2}, \qquad x \in X.
\end{equation}

It is well-known that the smoothness of a (finite) measure $\mu$
determines the boundary behavior of its Poisson integral
$\mu*P_\alpha$. On the other hand, in many cases the
radial convergence  of $\mu * P_\alpha$ may imply certain
regularity of $\mu,$ see \cite{Lo43}, \cite{Ge57}, \cite{Do63}, and
\cite{BrCh90} for a general account of tauberian results of this
type.

For instance, if one has in \eqref{poisson1}
$$\lim_{\alpha \to 0+}\bigl( D(\alpha +i\beta)x ,x
\bigr)=0, \qquad \beta \in (a,b),$$ then  $E\bigl((a,b)\bigr)=0,$
see for example \cite{Lo43}.

Thus, in the case of unitary $(T(t))_{t \in\mathbb R}$, one easily
obtains   the description of spectral subspaces of $A$
(or equivalently of $(T(t))_{t \in\mathbb R}$) in terms of the Poisson
integrals  $E * P_\alpha:$
\begin{eqnarray*}\label{hilb}
X(F)= \{x \in X: \lim_{\alpha \to 0+ }D(\alpha +i\beta)x=0, \,
\beta \in \mathbb R \setminus F\}.
\end{eqnarray*}
Furthermore, if the point spectrum $\sigma_p(A)$ of $A$ is empty, then
\begin{eqnarray*}
X(F)=\{x \in X: \sup_{\alpha
> 0} \| D(\alpha +i\beta)x\|< \infty, \, \beta \in \mathbb R
\setminus F\},
\end{eqnarray*}
 cf. Theorem
\ref{groups} below. Such descriptions are of value, for instance,
in scattering theory, see e.g. \cite[Chapter 3]{BaWo83}.

If $X$ is a Banach space and/or $C_0$-group $(T(t))_{t \in \mathbb
R}$ on $X$ is polynomially bounded, $E$ is  merely a distribution
although one still has the formula \eqref{poisson}. In this
setting one could not expect in general direct relations between
the boundary behavior of $D$ and the regularity of $E.$
Nevertheless, it appears that if $(T(t))_{t \in \mathbb R}$ does
not grow too fast and the Poisson integral of $E$ has zero radial
boundary values, then $E$ is forced to be zero, thus a
distributional counterpart  of tauberian theorems  in
\cite{Lo43}, \cite{Ge57} and \cite{Ru78} holds. More precisely,
 the following statement is true.

\begin{theorem}\label{groups}
  Let $X$ be a Banach space, and let $(T(t))_{t \in \mathbb R}$
be a $C_0$-group on $X$ with  generator $A,$ such that
$\|T(t)\|\le M (1 + |t|^a)$, $a \in [0,2]$. Then for any closed $F
\subset \mathbb R$
\begin{equation}\label{d1}
X(F)= \{x \in X:  \lim_{\alpha \to 0+ }D(\alpha +i\beta)x=0, \,
\beta \in \mathbb R \setminus F\}.
\end{equation}
If $X$ is reflexive,  $\sigma_p(A)=\emptyset$ and $a \in[0,2),$
then moreover
\begin{equation}\label{d2}
X(F)= \{x \in X:  \sup_{\alpha > 0 }\|D(\alpha +i\beta)x\|<
\infty, \, \beta \in \mathbb R \setminus F\}.
\end{equation}
\end{theorem}

Note that we require neither UMD property of the underlying
Banach space, nor the boundedness of the group. Thus a
well-developed spectral theory of (mostly) bounded groups on UMD
spaces  is by no means available in our situation.

 Theorem \ref{groups} can be used to give an algebraic
description of spectral subspaces of $(T(t))_{t \in \mathbb R},$
thus partially improving and extending \cite[Theorem 3.1]{CuNe89}
(see also \cite[Theorem 3.2 and Remark
3.5]{FoVa74}, \cite[Theorem 1.2]{Vr73}, \cite[Theorem 4]{MiMiNe}) to
linear (in general, unbounded) operators generating $C_0$-groups
of subquadratic growth. For a related result on the generators of
bounded groups see \cite[Corollary 6.7]{BaChTo02}.

\begin{theorem}\label{ranges}
Let $X$ be a Banach space, and let $(T(t))_{t \in
\mathbb R}$ be a $C_0$-group on $X$ with generator $A,$ such that
$\|T(t)\|\le C (1 + |t|^a)$.

If $a \in [0,1)$, then for any closed
$F\subset \mathbb  R$,
\begin{equation}\label{intersect}
X(F)=\bigcap_{\beta \in \mathbb  R \setminus F} {\rm ran} (i\beta-A)^{2}.
\end{equation}

If $a \in [1,2)$, then for any closed
$F\subset \mathbb  R$,
\begin{equation}\label{intersect1}
X(F)=\bigcap_{\beta \in \mathbb  R \setminus F} {\rm ran} (i\beta-A)^{3}.
\end{equation}
\end{theorem}

\begin{remark}
Note that the right hand side of \eqref{intersect} does not depend
on the exponent $a$ while the results in \cite{FoVa74},
\cite{Vr73}, \cite{CuNe89}, \cite{MiMiNe} formulated (for in
general unbounded operators) as in Theorem \ref{ranges} would only
yield
$$
X(F)=\bigcap_{\beta \in \mathbb R\setminus F} {\rm ran} (i\beta-A)^{a+2}.
$$
\end{remark}

We prove Theorem~\ref{te1} in Section 2. The (counter)-examples (Proposition~\ref{proposition}) are discussed
in Section 3. In Section 4 we formulate a maximum principle and a local half plane versions of Theorem~\ref{te1}.
Finally, in Section 5 we prove Theorems~\ref{groups} and \ref{ranges}.

\section{The proof of Theorem~\ref{te1}}
\label{sect}

First of all, using \cite[Section 1.3]{NIK} or \cite{BOR},
we replace conditions on $M_r(u)$ by the same conditions
on $M_r(|u|)$.

The scheme of our proof goes back at least to F.~Wolf (1941). Let us denote by $R$ the set of $\theta\in [0,2\pi]$ such that $u$ is
continuous in a neighborhood of $e^{i\theta}$ (in the unit disc).
Then $S=[0,2\pi]\setminus R$ is a closed subset of $[0,2\pi]$.
If $S$ is not empty, then, by the Baire category theorem, we can find
an open interval $I\subset [0,2\pi]$ such that
\begin{gather}
S\cap I\not=\emptyset,\label{s1}\\
\sup_{z\in \Omega^h(\varphi),\,\varphi\in S\cap I}|u(z)|\le A<\infty. \notag
\end{gather}
It remains to prove that for any closed interval $J\subset I$,
$u$ is bounded in the sector $\{re^{i\theta}:0\le r<1,\, \varphi\in J\}$.
After that, standard uniqueness results give us a contradiction with
\eqref{s1}.

Let $J_1\subset R\cap J$ be an open interval such that its endpoints $\varphi_1,\varphi_2$ are in $S$, and let $T$ be the
connected component of $\{re^{i\theta}:1/2< r<1,\, \theta\in J_1\}\setminus (\Omega^h(\varphi_1)\cup\Omega^h(\varphi_2))$
such that $\partial T\cap \partial \D\not=\emptyset$. Suppose that $u$ is bounded in every such $T$. Applying the maximum
principle,
$$
\sup_T u\le \sup_{\Omega^h(\varphi_1)\cup\Omega^h(\varphi_2)\cup D(0,1/2)}u
$$
(where $D(z,R)=\{w:|w-z|\le R\}$), we obtain then that $u$ is uniformly bounded in all $T$.

Thus, after a conformal map to the upper half-plane $\mathbb C_+$,
our problem is reduced to the following one.

Given a function $u$ harmonic in $Q=\{x+iy:|x|<1,\,0<y<1\}$, denote
$$
M_y=\max_{0<x<1}|u(x+iy)|.
$$

{\bf Reduced problem.} Suppose that $u$ is continuous up to
$(0,\delta)\subset\partial Q$, $u=0$ on $(0,\delta)$
for some $\delta>0$ and $u$ is bounded on $\Delta^h\cap Q$.

If (a) $h(t)=0$ and
$$
M_y=o(y^{-2}),\qquad y\to 0+,
$$

or (b) $h(t)=ct^3$, $c>0$, and
$$
\log^+M_y=o(y^{-1}),\qquad y\to 0+,
$$

or (c) $h(t)=ct$, $c>0$,  and
$$
\log^+M_y=o(y^{-\pi/(2\arctan (1/c))}),\qquad y\to 0+,
$$

or (d) $h(t)=t^\gamma$, $0<\gamma<1$,  and
$$
\log^+\log^+M_y=o(y^{\gamma-1}),\qquad y\to 0+,
$$
then $u$ is bounded in a neighborhood of the point $0$ in
$\{x+iy\in Q:x>0\}$.

Next we apply the Schwartz reflection principle and then consider the
function $v(z)$ equal to $u(1/z)$ in the cases (a) and (b), equal
to
$$
u\bigl(z^{-(2/\pi)\arctan (1/c)}\bigr)
$$
in the case (c) and equal to
$$
u\Bigl( \Bigl(\frac{2(1-\gamma)}{\pi\gamma} \log z\Bigr)^{\gamma/(\gamma-1)} \Bigr)
$$
in the case (d).

Then the function $v$ is harmonic
in $O=\mathbb C\setminus ((-\infty,0]\cup K)$ for some compact $K$,
$v=0$ on $O\cap (0,\infty)$,
\begin{gather}
\lim_{|y|\to\infty,\,-\varphi(y)\le x\le 0}v(x+iy)=0,\notag\\
|v(re^{i\theta})|\le \psi(r,\theta), \qquad -\pi/2\le \theta\le \pi/2,\, r>0,
\label{es2}
\end{gather}
for certain functions $\varphi,\psi$, and we need only to verify
that $v$ is bounded at infinity in the right half-plane $\Pi$.

Here in the case (a) $\varphi=0$, $\psi(r,\theta)=o(r^2/|\theta|^2)$, $r\to\infty$,
in the case (b) $\varphi(y)=k/|y|$, $\log^+\psi(r,\theta)=o(r/|\theta|)$, $r\to\infty$,
in the case (c) $\varphi(t)=kt$, $\log^+\psi(r,\theta)=o(r/|\theta|^{k_1})$, $r\to\infty$,
and in the case (d) $\varphi(t)=t$, $\log^+\log^+\psi(r,\theta)=o(\log r/|\theta|^{1-\gamma})$,
$r\to\infty$, for some positive numbers $k,k_1$
depending on $c$.

To get rid of the singularities in the estimates \eqref{es2} we use the $\log$-$\log$ theorem
of Levinson-Sj\"oberg. Namely, we consider the functions $v_n$, $v_n(z)=v(2^nz)$.
The functions $f_n$ analytic on $H=\{x+iy:1/2\le x\le 4,\,|y|\le 2\}$ are determined
by the relations $\Re f_n=v_n$, $f_n(1+i)=v_n(1+i)$.

We use the fact that given a function $v$ harmonic in the disc $D(z,R)$ and its conjugate function $\tilde{v}$,
for an absolute positive constant $c$ we have
\begin{equation}
\|\grad \tilde v(z)\|\le \frac cR \sup_{w_1,w_2\in D(z,R)}|u(w_1)-u(w_2)|.
\label{x111}
\end{equation}
Then the functions $f_n$ satisfy the following estimates on $H$:
\begin{align*}
{\rm (a)}\quad&|f_n(x+iy)|=\frac{o(2^{2n})}{|y|^2},\qquad n\to\infty,\\
{\rm (b)}\quad&\log^+|f_n(x+iy)|=\frac{o(2^{n})}{|y|},\qquad n\to\infty,\\
{\rm (c)}\quad&\log^+|f_n(x+iy)|=\frac{o(2^{n})}{|y|^{k_1}},\qquad n\to\infty,\\
{\rm (d)}\quad&\log^+\log^+|f_n(x+iy)|=\frac{o(n)}{|y|^{1-\gamma}},\qquad n\to\infty.
\end{align*}

Next, we apply the following quantitative version of the $\log$-$\log$ theorem
(see \cite[Section 3]{DOM}).

\begin{theorem} {\rm(Y.~Domar, see the argument in \cite[pp. 376--379]{KOO})} Let $w:(0,2)\to[1,\infty)$ be a decreasing function.
If $h$ is subharmonic on $D$, $h(x+iy)\le w(|y|)$, $x+iy\in H$, and
$$
\sum_{k\ge 0}w^{-1}(2^kT)\le \frac 1{10}
$$
for some $T>0$, then
$$
h(z)\le 2T,\qquad z\in H_0=\{x+iy:1\le x\le 2,\,|y|\le 1\}.
$$
\end{theorem}

Applying this theorem to $\log^+|f_n|$ we obtain
\begin{align*}
{\rm (a)}\quad&\sup_{H_0}|v_n|=o(2^{2n}),\qquad n\to\infty,\\
\text{\rm (b)-(c)}\quad&\sup_{H_0}\log^+|v_n|=o(2^n),\qquad n\to\infty,\\
%{\rm (c)}\quad&\sup_{D_0}\log^+|v_n|=o(2^n),\qquad n\to\infty,\\
{\rm (d)}\quad&\sup_{H_0}\log^+\log^+|v_n|=o(n^{1/\gamma}),
\qquad n\to\infty.
\end{align*}

Thus, the function $v$ satisfies the following estimates
in $\Pi$:
\begin{align*}
{\rm (a)}\quad&|v(z)|=o(z^2),\qquad |z|\to\infty,\\
\text{\rm (b)-(c)}\quad&\log^+|v(z)|=o(z),\qquad |z|\to\infty,\\
{\rm (d)}\quad&\log^+\log^+|v(re^{i\theta})|=
\left\{
\begin{gathered}
o\Bigl(\frac{\log r}{|\theta|^{1-\gamma}}\Bigr),\quad |\theta|\le\frac\pi2,\\
o((\log r)^{1/\gamma}),\quad |\theta|\le\frac\pi4,
\end{gathered}
\right.
\,\quad r\to\infty.
\end{align*}

In the cases (b)-(d) we fix $c\in O$ and define an analytic function
$f$ by the relations $\Re f=v$, $f(c)=v(c)$. Again by \eqref{x111} we obtain that
\begin{align}
\text{\rm (b)-(c)}\quad&
|f(iy-a/(2|y|))|=o(y^2),\qquad |y|\to\infty,\notag\\
\text{\rm (b)-(c)}\quad&
\log^+|f(iy-s)|=o(y),\qquad |y|\to\infty,\,\, 0\le s\le a/(2|y|),\notag\\
\text{\rm (b)-(c)}\quad&
\log^+|f(z)|=o(z),\qquad z\in\Pi,\,|z|\to\infty,\notag\\
{\rm (d)}\quad&
|f(iy-1)|=o(y),\qquad |y|\to\infty,\label{s25}\\
{\rm (d)}\quad&
\log^+|f(iy-s)|=o(y),\qquad |y|\to\infty,\,\, 0\le s\le 1,\label{s28}\\
{\rm (d)}\quad&\log^+\log^+|f(re^{i\theta})|=
\left\{
\begin{gathered}
o\Bigl(\frac{\log r}{|\theta|^{1-\gamma}}\Bigr),\quad |\theta|\le\frac\pi2,\\
o((\log r)^{1/\gamma}),\quad |\theta|\le\frac\pi4,
\end{gathered}
\right.
\,\quad r\to\infty,\notag
\end{align}
where $a=k$ in the case (b) and $a=1$ in the case (c).

By a Phragm\'en-Lindel\"of type theorem (using the argument of \cite[III C]{KOO}), in the case (b)-(c) we obtain that
$$
|f(z)|=o(z^2),\qquad z\in\Pi,\,|z|\to\infty.
$$

In the case (d), for sufficiently large closed disc $U=\overline{D(0,R)}$, we apply the following lemma to
the function $w(z)=\log^+|f(z-1)/(z+1)^2|$ subharmonic in $\Omega=\Pi\setminus U$, where $\Pi$
is the right half plane.

\begin{lemma} Let $w$ be subharmonic in $\Omega$ and upper semicontinuous in $\overline{\Omega}$,
$w\le 0$ on $\partial\Omega$, $0<\gamma<1$, and let
\begin{equation}
\log^+|w(re^{i\theta})|=
\left\{
\begin{gathered}
o\Bigl(\frac{\log r}{|\theta|^{1-\gamma}}\Bigr),\quad
|\theta|\le\frac\pi2,\\
o((\log r)^{1/\gamma}),\quad |\theta|\le\frac\pi4,
\end{gathered}
\right.
\,\quad r\to\infty.
\label{s21}
\end{equation}
Then $w\le 0$ in $\Omega$.
\end{lemma}

\begin{proof} First of all we choose $\delta>0$ such that
\begin{equation}
\exp[(n-1)^2]\ge 2\delta n\exp[(1-\gamma)(n+1)^2],\qquad n\ge 1.
\label{s23}
\end{equation}

We fix $z\in\Omega$. For large $N,M$ we consider the function
$$
\beta(x)=\left\{
\begin{gathered}
e^{N-1},\qquad 0\le x \le e^N,\\
x\exp\Bigl[-\Bigl(\frac{\log x}N\Bigr)^2\Bigr],\qquad x>e^N,
\end{gathered}
\right.
$$
the domain
$$
\Sigma=\{x+iy\in\Omega:0<x<e^{NM},\,|y|<\beta(x)\},
$$
and the sets
\begin{align*}
S_0=&\{x+iy\in\partial\Sigma:0<x<e^{N}\},\\
S_n=&\{x+iy\in\partial\Sigma:e^{nN}<x<e^{(n+1)N}\},\qquad 1\le n<M,\\
S_M=&\{x+iy\in\partial\Sigma:x=e^{NM}\}.
\end{align*}

Then by the theorem on harmonic estimation (see \cite[p.256]{KOO})
we have
\begin{equation}
u(z)\le \sum_{0\le n\le M}\omega(z,S_n,\Sigma)\cdot \sup_{S_n}w ,
\label{s22}
\end{equation}
where $\omega(z,S,\Sigma)$ is harmonic measure of
$S\subset\partial\Sigma$ with respect to $z$ in $\Sigma$.

An easy geometric argument shows that for $N\ge N(z)$ we have
$$
\omega(z,S_0,\Sigma)+\omega(z,S_1,\Sigma)\le e^{-N/2}.
$$
Indeed, it suffices to estimate the left hand side expression from above  by
\begin{multline*}
\omega(z,\partial D(0,e^{N-1})\cap \Pi,D(0,e^{N-1})\cap \Pi)\\=
\omega(ze^{1-N},\partial D(0,1)\cap \Pi,D(0,1))-
\omega(ze^{1-N},\partial D(0,1)\setminus \Pi,D(0,1))\\
\asymp c|z|e^{-N},\qquad N\to\infty.
\end{multline*}

Furthermore, the Ahlfors--Carleman theorem (see, for instance, %\cite[p.103]{KOO2},
\cite[p.148, Theorem 6.1]{GM})
shows that for large $N$ we have
\begin{multline*}
\omega(z,S_n,\Sigma)\le \frac 8\pi\exp\biggl(
-\pi\int^{\exp(nN)}_{\exp((n-1)N)}\frac{dr}
{r\exp\Bigl[-\Bigl(\frac{\log r}N\Bigr)^2\Bigr]}
\biggr)\\ \le
\exp[-Ne^{(n-1)^2}],\qquad 2\le n\le M.
\end{multline*}

Next we use that by \eqref{s21}, for large $N\ge N(\delta)$, the following estimates are fulfilled:
\begin{align*}
\sup_{S_0\cup S_1}w\le &e^{N/2},\\
\sup_{S_n}w\le &\exp\Bigl[\delta nN\exp
\bigl[(1-\gamma)(n+1)^2\bigr]\Bigr],\qquad 1<n<M,\\
\sup_{S_M}w\le &\exp[(NM)^{1/\gamma}].
\end{align*}

By \eqref{s22}, we obtain that
\begin{multline*}
w(z)\le 1+\exp[(NM)^{1/\gamma}-Ne^{(M-1)^2}]
\\+
\sum_{2\le n<M}\exp\Bigl[\delta nN\exp
\bigl[(1-\gamma)(n+1)^2\bigr]-Ne^{(n-1)^2}\Bigr].
\end{multline*}
For large $N\ge N(\delta)$, $M\ge M(N,\gamma)$, using \eqref{s23} we conclude that
$$
w(z)\le 2+\sum_{2\le n<M}\exp\Bigl[-\delta nN\exp
\bigl[(1-\gamma)(n+1)^2\bigr]\Bigr]\le 3.
$$
Thus, $w$ is bounded on $\Omega$,
and hence, $w\le 0$ on $\Omega$.
\end{proof}

As a result, in the case (d) we obtain
$$
|f(z)|=O(z^2),\qquad z\in\Pi,\,\,|z|\to\infty,
$$
and the Phragm\'en-Lindel\"of theorem together with \eqref{s25},
\eqref{s28} gives
us that
$$
|f(z)|=o(z),\qquad z\in\Pi,\,\,|z|\to\infty.
$$

Finally, in the cases (a)-(d) the harmonic function $v$ is bounded at infinity
on $i\mathbb R\cup\mathbb R_+$ and
$$
|v(z)|=o(z^2),\qquad z\in\Pi,\,\,|z|\to\infty.
$$
Therefore, $v$ is bounded in $\Pi$ and our proof is completed.

\section{Examples}

\begin{proof}[Proof of Proposition~\rm\ref{pp}]
{\rm (a)} If $h(t)=o(t^3)$, $t\to 0$, then
the function $u(re^{i\theta})=\sum_{n>0}nr^n\sin(n\theta)$
is harmonic in $\mathbb D$,
$$
M_r(u)=O((1-r)^{-2}),\qquad r\to 1-,
$$
and
$$
\lim_{z\in \Omega^h(\varphi),\,z\to e^{i\varphi}}u(z)=0, \qquad
%\text{for every} \,\,
\varphi\in [0,2\pi],
$$

Next, we use the construction from \cite[Appendix]{BoChTo07}. It is
valid in the case (c) and works with small modifications
in the cases (b), (d). For the sake of simplicity, we give here just a local half plane example in the case (b).
Namely, if $1<\gamma\le 3$, $h(t)=t^\gamma$, $\varepsilon>0$,
then there exists a function $u\not=0$ harmonic in $Q=\{x+iy:|x|<1,\,0<y<1\}$ and
such that
\begin{align*}
\log^+|u(x+iy)|&\le\frac{\varepsilon}{y},\\
\lim_{z\in Q,\,z\to x}u(z)&=0,\qquad x\in [-1,1]\setminus \{0\},\\
\lim_{z\in \Delta^h,\,z\to 0}u(z)&=0.
\end{align*}

{\bf Construction.} Let $\max(0,2-\gamma)<\delta<1$, and let
$$
f_0(z)=\exp\Bigl(\frac {\varepsilon}z-\frac 1{z^\delta}\Bigr),\qquad z\in \overline{Q}\setminus \Delta^h,\,\Re z>0,
$$
where the branch of $z^\delta$ is chosen to be positive on the positive half axis. Then
\begin{align*}
\Im f_0(x)&=0,\qquad x\in (0,1),\\
|f_0(x+iy)|&\le\exp\frac{\varepsilon}{y},\qquad x+iy\in Q,\\
|f_0(z)|&\le\exp(-c|z|^{-\delta}),\quad z=x+iy,\,y^\gamma<x<2y^\gamma,\,0<y<1.
\end{align*}
Next, fix $f_1\in C^2([1,2])$ with $0\le f_1\le 1$, such that $f_1$
vanishes in a neighborhood of the point $1$, and equals $1$ in a
neighborhood of the point $2$. Define $f_2(z)$, $z=x+iy\in Q$, by
$f_2(z)=f_0(z)$, $x>2y^\gamma$, $f_2(z)=f_0(z)f_1(t)$,
$x=ty^\gamma$, $1\le t\le 2$, $f_2(z)=0$, $x<y^\gamma$. We have
\begin{align*}
f_2&\in C^2(\overline{Q}\setminus\{0\}),\\
\Im f_2&\not\in L^\infty(Q),\\
\bar\partial f_2&\in L^\infty(Q),\\
\lim_{z\to 0}\bar\partial f_2(z)&=0,\\
\Im f_2(z)&=0,\qquad z\in\Delta^h\cup [-1,1]\setminus \{0\},\\
|f_2(x+iy)|&\le C\exp\frac{\varepsilon}{y}.
\end{align*}
Now we define
$$
f_3(z)=\frac 1\pi \int_Q \frac{\bar\partial
f_2(\zeta)}{z-\zeta}\,dm_2(\zeta),
$$
where $m_2$ is the planar Lebesgue measure. Then $f_3\in
C(\overline{Q})\cap C^1(Q)$, and $\bar\partial f_3=\bar\partial
f_2$ on $Q$. Let $f_4$ be a function harmonic in $Q$ and
continuous in $\overline{Q}$ such that $f_4=\Im f_3$ on $[-1,1]$.
Denote
$$
u=\Im (f_2-f_3)+f_4.
$$
Then $u$ is harmonic in $Q$, $u\not=0$, $u\in
C(\overline{Q}\setminus\{0\})$, $u\upharpoonright [-1,1]\setminus
\{0\}=0$, $|u(x+iy)|\le C+C\exp(\varepsilon/y)$, and
$$
\lim_{z\in\Delta^h,\, z\to 0}u(z)=\lim_{z\in\Delta^h,\, z\to 0}f_4(z)-\Im f_3(z)=0.
$$
\end{proof}

If $h$ decreases at zero slower than a power of $t$, then the
critical growth rate is bigger than in Theorem~\ref{te1} (d).
However, the growth
$$
\int^1\log^+\log^+M_r(u)\,dr=\infty,
$$
does not correspond to any approach domain. One encounters the same critical
growth boundary for the existence of both  MacLane
asymptotical and Beurling generalized distributional boundary
values for analytic functions (see \cite{VOL} for a discussion).

Here we have the following result.

\begin{remark}
If $f$ is a sufficiently regular positive function on $(0,\infty)$ such that
\begin{equation}
\left\{
\begin{gathered}
\lim_{x\to \infty}f(x)=\infty,\\
\int^\infty\frac{f(x)}{x^2}\,dx=\infty,
\end{gathered}
\right.
\label{e2}
\end{equation}
and $h$ is an approach function, then there exists a
function $u\not=0$ harmonic in $\mathbb D$
such that
$$
\log^+\log^+M_r(u)=O(f((1-r)^{-1}))),\qquad r\to 1-,
$$
and
$$
\lim_{z\in \Omega^h(\varphi),\,z\to e^{i\varphi}}u(z)=0, \qquad
%\text{for every} \,\,
\varphi \in [0,2\pi],
$$
\end{remark}

Let us sketch a construction of such a function (cf. \cite[Example 3.3]{PLM}). It suffices
to find a sufficiently regular increasing function $\beta:[0,1]\to[0,1]$,
$\beta(0)=0$, such that
\begin{equation}
\left\{
\begin{gathered}
\beta(h(t))=o(t),\qquad t\to 0,\\
\int_t\frac{ds}{\beta(s)}=o(f(1/\beta(t))),\qquad t\to0.
\end{gathered}
\right.
\label{e1}
\end{equation}
If $\Omega=\{x+iy:0<x<1,\,|y|<\beta(x)\}$, and $\Phi$ is a
conformal map of $\Omega$ onto $\mathbb C_+$ such that
$\Phi(0)=\infty$, $\Phi((0,1))=i\mathbb R_+$, then we consider the
analytic function $\exp(-i\Phi)$. The estimates by Warschawski on
the asymptotics of $\Phi$ together with properties \eqref{e1} give
us a possibility to use the scheme of \cite[Appendix]{BoChTo07}.
Finally, relations \eqref{e2} guarantee the existence of $\beta$
satisfying \eqref{e1}.

\section{Maximal and local half plane versions of Theorem~\ref{mmm}}

An easy modification of the argument in Section~\ref{sect} gives us the following maximum principle, cf. \cite{BoChTo07,WOL}.
For a different set of problems in this direction see \cite{Va08}.

\begin{theorem}
Let $u$ be harmonic in the unit disc $\mathbb D$ and let
\begin{equation*}
\limsup_{z\in \Omega^h(\varphi),\,z\to e^{i\varphi}}|u(z)|\le 1,
\qquad
%\text{for every} \, \,
\varphi \in [0,2\pi].
\end{equation*}

{\rm (a)} If $h(t)=0$ and
$$
M_r(u)=o((1-r)^{-2}),\qquad r\to 1-,
$$
then $|u|\le 1$ in $\mathbb D$.

{\rm (b)} If $\lim_{t\to 0}h(t)t^{-3}=\infty$ and
$$
\log^+M_r(u)=o((1-r)^{-1}),\qquad r\to 1-,
$$
then $|u|\le 1$ in $\mathbb D$.

{\rm (c)} If $h(t)=ct$, $c>0$,  and
$$
\log^+M_r(u)=o((1-r)^{-\pi/(2\arctan (1/c))}),\qquad r\to 1-,
$$
then $|u|\le 1$ in $\mathbb D$.

{\rm (d)} If $h(t)=t^\gamma$, $0<\gamma<1$,  and
$$
\log^+\log^+M_r(u)=o((1-r)^{\gamma-1}),\qquad r\to 1-,
$$
then $|u|\le 1$ in $\mathbb D$.
\end{theorem}

Next we formulate a local half plane version of Theorem~\ref{mmm} (a), (b) which will be needed in the next section.

Define the rectangle
$$
Q:=\{z=x+iy: -1<x<1,\, 0<y<1\},
$$
and for every $x \in (-1,1)$ and the approach function $h$
define $\Delta^h(x):=\Delta^h+x$.

Let $M_y(u):=\sup_{x \in (-1,1)} |u(x+iy)|$.

\begin{theorem} \label{halfplane}
Let $u$ be a harmonic function on $Q$.
Assume that
\begin{equation}
\lim_{z\in \Delta^h(x),\,z\to x}u(z)=0, \qquad
%\text{for every}\,\,
x\in [-1,1]. \label{lim1}
\end{equation}
If either

 {\rm (a)}  $h(t)=0$ and
$$
M_y (u)={\rm o}(y^{-2}),\qquad y\to 0+,
$$
 or

{\rm (b)}  $h(t)=ct^3$, $c>0$, and
$$
\log^+M_y (u)={\rm o}(y^{-1}),\qquad y\to 0+,
$$
then  the function $u$ extends continuously to the interval $(-1,1)$, and $u=0$ on $(-1,1)$.
\end{theorem}

The proof is similar to that of Theorem~\ref{mmm}.

\section{An operator theoretical application}

In this section we derive  Theorems \ref{groups} and \ref{ranges}
stated in Introduction as corollaries of our local uniqueness result,
Theorem \ref{halfplane}.

If $X$ is a Banach space, and  $(T(t))_{t \in \mathbb R}$ is a
$C_0$-group on $X$ with the generator $A$ such that $\|T(t)\|\le M
(1 + |t|^a)$, $a \ge 0$, then the resolvent $R(\lambda, A), \lambda
\in \mathbb C \setminus i\mathbb R,$ is the Carleman transform of
$(T(t))_{t \in \mathbb R}$,
\begin{equation}
R(\lambda, A) := \left\{ \begin{array}{ll}
\int_0^\infty e^{-\lambda t} T(t) \; dt , & {\rm Re}\, \lambda >0 ,\\[2mm]
-\int_{-\infty}^0 e^{-\lambda t} T(t) \; dt , & {\rm Re}\, \lambda
<0,
\end{array} \right.
\end{equation}
thus
\begin{equation}\label{estimate}
\|R(\lambda, A)\| \le \frac{{M}}{|{\rm Re}\, \lambda|^{a+1}},
\quad {\rm Re} \, \lambda \neq 0.
\end{equation}

The estimate \eqref{estimate} and the resolvent identity allow us to extend the property of horizontal convergence of
$D(\alpha+i\beta)=R(\alpha+i\beta,A)-R(-\alpha+i\beta,A)$ to the regions
$$
\Omega^h(\beta):=\bigl \{\alpha+i\beta': \alpha \in (0,1),\,|\beta -\beta'|\le h(\alpha)  \bigr \},
$$
corresponding to the approach function $h(t)=t^{a+1}$.

\begin{lemma}\label{resolvent}
Let $X$ be a Banach space, and let $x \in X$ be fixed.
\begin{itemize}
\item [(i)]
$\lim_{\alpha \to 0+}\|D(\alpha +i\beta)x\|=0$ if and only if
\newline\noindent $\lim_{z\to i\beta,\, z \in\Omega^h(\beta)} \|D(\alpha +i\beta)x\|=0$.

\item [(ii)] If $\sigma_p(A)=\emptyset$, then the limit
$\lim_{\alpha \to 0+}D(\alpha+i\beta)x$ exists if and only if
$\lim_{\alpha \to 0+}\|D(\alpha +i\beta)x\|=0$.

\item [(iii)] If $X$ is reflexive and $\sigma_p(A)=\emptyset$, then
$\sup_{\alpha > 0} \|D(\alpha + i\beta)x\|<\infty$ if
and only if $D(\alpha +i\beta)x$ tends to $0$ weakly as $\alpha \to 0+$.

\end{itemize}
\end{lemma}
\begin{proof}
(i) Fix $x$ and $\beta$ such that $\lim_{\alpha \to 0+}\|D(\alpha +i\beta)x\|=0$.
We use that
\begin{gather}
D(\alpha+i\beta)=2\alpha[\alpha^2-(A-i\beta)^2]^{-1},\label{de1}\\
(A-i\beta)D(\alpha+i\beta)=\alpha[R(\alpha+i\beta,A)+R(-\alpha+i\beta,A)].\label{de2}
\end{gather}
Therefore,
\begin{gather*}
D(\alpha+i\beta)-D(\alpha+i\beta_1)\\=(2\alpha)^{-1}
\bigl( [\alpha^2-(A-i\beta_1)^2]-[\alpha^2-(A-i\beta)^2]\bigr)D(\alpha+i\beta)D(\alpha+i\beta_1)\\
=(2\alpha)^{-1}[2i(A-i\beta_1)(\beta_1-\beta)-(\beta_1-\beta)^2]D(\alpha+i\beta)D(\alpha+i\beta_1)\\
=i(\beta_1-\beta)[R(\alpha+i\beta_1,A)+R(-\alpha+i\beta_1,A)]D(\alpha+i\beta)\\-
\frac{(\beta_1-\beta)^2}{2\alpha}D(\alpha+i\beta)D(\alpha+i\beta_1).
\end{gather*}
Since $(\beta_1-\beta)^2\alpha^{-1}\to 0$ and
$|\beta_1-\beta|\cdot
\|R(\alpha+i\beta_1,A)+R(-\alpha+i\beta_1,A)\|$ is bounded when
$\alpha\to 0+$, $\alpha + i \beta_1 \in \Omega^h (\beta)$, we
conclude that
$$
\|D(\alpha +i\beta_1)x\|\to 0, \qquad \alpha\to 0+, \,\alpha + i
\beta_1 \in \Omega^h (\beta).
$$

(ii) Let $\lim_{\alpha \to 0+}D(\alpha +i\beta)x=y$. By \eqref{de2},
$$
(A-i\beta) D(\alpha+i\beta)x=\alpha (R(\alpha+i\beta,A)x+ R(-\alpha+i\beta,A)x.
$$
Therefore, $D(\alpha+i\beta)x \in {\rm dom}\, (A-i\beta)$. By \eqref{de1},
$$
(A-i\beta)^2 D(\alpha+i\beta)x =-2\alpha x  +\alpha^2 D(\alpha+i\beta)x \to 0,\qquad \alpha \to
0+.
$$
Since $(A-i\beta)^2$ is closed, $y \in {\rm dom} \, (A-i\beta)^2$,
$(A-i\beta)^2y=0$ and then, by assumption,  $y=0$.

 (iii) Since $X$ is reflexive, the set
$S:=\{D(\alpha_n +i\beta)x: n \in \mathbb N\}$ is weakly
precompact for any sequence $\alpha_n \searrow 0$, $n \to \infty$.

 Let $y \in X$ be such that
$$
w-\lim_{n \to \infty} D(\alpha_{n_k} +i\beta)x=y
$$
for a subsequence $\alpha_{n_k} \searrow 0$, $k \to \infty$. As above,
$$
(A-i\beta)^2 D(\alpha_{n_k}+i\beta)x=-2\alpha_{n_k} x
+\alpha_{n_k}^2 D(\alpha_{n_k}+i\beta)x \to 0.
$$
Since $(A-i\beta)^2$ is closed, it is also weakly closed, hence  $(A-i\beta)^2y=0$, and then $y=0$.
Thus the only weak (sequential) limit point of $S$ is $0$ so that
the weak limit of $D(\alpha_{n} +i\beta)x$, $n \to \infty$, is $0$.
Since the choice of $\{\alpha_n\}$ was arbitrary, the assertion of the lemma
follows.
\end{proof}

Theorem \ref{groups} now follows from Theorem \ref{halfplane}.

\begin{proof}[Proof of Theorem \ref{groups}]
We start with the proof of \eqref{d1}.
Let
$$
X^D(F):= \{x \in X: \lim_{\alpha \to 0+ }D(\alpha +i\beta)x=0, \,
\, \beta \in \mathbb R \setminus F \}.
$$
Obviously, $x \in X(F)$ implies $x \in X^D(F)$ since the function
$R(\lambda, A)x$ extends analytically through $i(\mathbb
R\setminus F)$ by the definition of $\sigma (x)$. To prove the
converse inclusion fix $x \in X^D(F)$ and note that
$$
\mathbb R\setminus F=\cup_{n \ge 1} (a_n,b_n), \qquad -\infty \le a_n
\le b_n \le +\infty.
$$
Let further $x^* \in X^*$ and $n \in
\mathbb N$ be fixed, and assume without loss of generality that
$[ia_n,ib_n]$ is compact. The harmonic function $\langle
D(\lambda)x, x^*\rangle$ defined on the rectangle $R_n=\{\lambda
\in \mathbb C: -1 \le {\rm Re}\, \lambda \le 1, \, b_n <  {\rm
Im}\, \lambda < a_n \}$ satisfies the estimate \eqref{estimate}.
By Theorem \ref{halfplane} (b) and Lemma \ref{resolvent} (i), $\langle
D(\lambda)x, x^*\rangle$ extends continuously to the interval
$(ia_n, ib_n )$ and is zero there. Arguing as in the proof of
\cite[Theorem 5.4]{BoChTo07} we obtain that
$\langle R(\lambda,A)x, x^*\rangle$ extends analytically to $R_n$. Furthermore, by
the uniform boundedness principle,
$$
\sup_{\lambda_1,\lambda_2 \in R_n \setminus i\mathbb R, \lambda_1\neq\lambda_2}\left \|\frac{R(\lambda_1,A)x -
R(\lambda_2,A)x}{\lambda_1-\lambda_2}\right \|<\infty,
$$
so that $R(\lambda,A)x$ is uniformly continuous on $R_n\setminus i\mathbb
R$ and therefore extends continuously to $R_n.$ The analyticity of
$\langle R(\lambda,A)x, x^* \rangle$ in $R_n$  for every $x \in
X^*$  implies that $R(\lambda, A)x$ is analytic in $R_n\setminus
i\mathbb R$ so that $R(\lambda,A)x$ is analytic in $R_n.$ Since
$n$ was arbitrary, the proof is finished.

Let now $X$ be reflexive,  $\sigma_p(A)=\emptyset$, and $a \in[0,2)$. Define
$$
X^D_b(F):= \{x \in X: \sup_{\alpha >0}\|D(\alpha +i\beta)x\|<\infty,
\, \beta \in \mathbb R \setminus F \}.
$$
Then arguing as above  $X(F) \subset X^D_b(F)$. Moreover, the
argument used above shows that to verify the opposite inclusion
it suffices to prove that for $x\in X^D_b$ and $x \in X^*$,
the function $\langle D(\lambda) x,x^*\rangle$ satisfies the conditions
of Theorem \ref{halfplane} for the approach function $h(t)=t^{a+1+\epsilon}$,
$\epsilon \in(0,2-a)$. Since $\sup_{\alpha >0} \|D(\alpha+i\beta)x\|<\infty,$
by Lemma \ref{resolvent} (iii) and the equality
\begin{gather*}
\bigl\langle (D(\alpha+i\beta)-D(\alpha+i\beta_1))x,x^*\bigr\rangle\\
=i(\beta_1-\beta)\bigl\langle D(\alpha+i\beta)x,[R(\alpha+i\beta_1,A)+R(-\alpha+i\beta_1,A)]^*x^*
\bigr\rangle\\
-\frac{(\beta_1-\beta)^2}{2\alpha}\bigl\langle D(\alpha+i\beta)x, D(\alpha+i\beta_1)^*x^*\bigr\rangle,
\end{gather*}
obtained in the proof of Lemma \ref{resolvent} (i), it follows
that $\langle D(\lambda) x, x^*\rangle$ satisfies \eqref{lim1}.
The verification of the condition (b) of Theorem \ref{halfplane} is
straightforward, thus \eqref{d2} is proved as well.
\end{proof}

\begin{remark}
We do not know  whether Theorem \ref{groups} is optimal with
respect to the exponent $a$. While, by Proposition \ref{proposition},
we cannot have a counterpart of Theorem \ref{halfplane} in the situation where $a>2$,
the corresponding examples of $C_0$-groups are still out of reach.
\end{remark}

\begin{proof}[Proof of Theorem \ref{ranges}]
First, we show that for every $n\ge 1$,
\begin{equation}\label{incl}
X(F) \subset \bigcap_{\beta \in \mathbb  R \setminus F} {\rm ran}(i\beta-A)^{n}.
\end{equation}
If $x \in X(F)$, then the function $f(\lambda)=R(\lambda, A)x$,
$\lambda \in \mathbb C\setminus i\mathbb R$, extends analytically to
$\mathbb C\setminus i F$ as well as its derivatives, and we have
$$
f^{(k)}(\lambda)=(-1)^k k! R^{k+1}(\lambda, A)x, \qquad 1 \le k \le n.
$$
Since for any $\beta \in \mathbb R\setminus F$ and
$\lambda \in \mathbb C \setminus i\mathbb R$,
\begin{equation*}\label{equal}
(i\beta-A)^n R^n (\lambda, A)x=x
+\sum_{k=1}^{n}C^k_n(i\beta-\lambda)^{k} R(\lambda, A)^k x,
%, \quad
%\lambda \in \mathbb C \setminus i\mathbb R, n \in \mathbb N,
\end{equation*}
by the closedness of $(i\beta - A)^n$ we obtain
$$
(i\beta - A)^n f(i\beta)=x, \qquad \beta \in \mathbb C \setminus F,
$$
so that $$x \in \bigcap_{\beta \in \mathbb R\setminus F} {\rm ran}
(i\beta - A)^n.$$

To prove the equality in \eqref{intersect1}, we note that $x\in{\rm ran} (i\beta-A)^3$,
$\beta \in \mathbb R\setminus F$, imply for $a\in[1,2)$
that $\lim_{\alpha \to 0+}D(\alpha +i\beta)x=0$
Indeed, by \eqref{de1}, \eqref{de2},
\begin{gather*}
D(\alpha+i\beta)(A-i\beta)^3y=\alpha^2 D(\alpha+i\beta)(A-i\beta)y-2\alpha (A-i\beta)y\\=
\alpha^3[R(\alpha+i\beta,A)+R(-\alpha+i\beta,A)]y-2\alpha (A-i\beta)y\to 0+,\,\, \alpha \to
0+.
\end{gather*}
Then, by Theorem \ref{groups},
$R(\lambda, A)x$ extends analytically to $\mathbb C \setminus F$, and the statement follows.
The proof of \eqref{intersect} is analogous.
\end{proof}

\begin{remark} In fact, to prove \eqref{intersect} we could just use Theorem~\ref{mmm} (a), and deal only with
radial behavior of the resolvents.
\end{remark}

\begin{remark}
Observe that  in \eqref{intersect}, \eqref{intersect1} one cannot replace ${\rm ran}
(i\beta - A)^n$, $n=2,3$, by ${\rm ran} (i\beta - A),$ see \cite[p.
136]{BaChTo02}. We do not know, however, whether one can replace
the exponent $2$ in \eqref{intersect} by a real number smaller
than $2.$
\end{remark}

Theorems \ref{groups} and \ref{ranges} apply in particular to
$C_0$-groups with (at most) linear growth. As a natural example of
such groups we mention $C_0$-groups $(\mathcal T(t))_{t \in
\mathbb R}$ on a Banach space $\mathcal X= X\oplus X$ generated by
the triangular operator matrices of the form
\begin{equation*}
 {\mathcal A}=\left( \begin{array}{cc}
       A_1  & B \\
       O & A_2 \end{array}
       \right)
 \end{equation*}
where $A_1$ and $A_2$ are the generators of bounded $C_0$-groups
$(T_1(t))_{t \in \mathbb R}$ and $(T_2(t))_{t \in \mathbb R}$ on
$X,$ and $B$ is bounded from ${\rm dom}(A_2)$ to ${\rm dom}(A_1).$
The operator ${\mathcal A}$ can be treated as an off-diagonal
perturbation of the generator of a bounded $C_0$-group defined by
the matrix with diagonal entries $A_1$ and $A_2$. The
corresponding  semigroup $({\mathcal T}(t))_{t \ge 0}$ is given by
 \begin{equation*}
{\mathcal T}(t)=\left( \begin{array}{cc}
        T_1(t) & \int_{0}^t T(s)B T(t-s)\, ds \\
       O & T_2(t) \end{array}
      \right), \qquad t \ge 0.
 \end{equation*}
Such matrices appear frequently in applications, e.g. in the study
of  second order abstract and concrete Cauchy problems. For more
details on this subject see e.g. \cite{Na89}.

Note that results similar to Theorems \ref{groups} and
\ref{ranges}  hold also for discrete groups $(T^n)_{n \in \mathbb
Z} \subset {\mathcal L}(X)$ such that $\|T^n\| \le M (|n|+1)^{a}, a \in
[0,2), n \in \mathbb Z$. Their proofs are straightforward
modifications of the proofs of Theorems \ref{groups} and
\ref{ranges} using the same boundary uniqueness statements of Theorem~\ref{halfplane}.
We leave formulation of these results to the interested reader.


\begin{thebibliography}{99}
\bibitem{AlGuPe07} J.  Alvarez, M. Guzm\' an-Partida, and S. P\'erez-Esteva,
\emph{Harmonic extensions of distributions,} Math. Nachr.
\textbf{280} (2007), 1443--1466.

\bibitem{AmMoGe96} W. O. Amrein, A. Boutet de Monvel,
and V. Georgescu, \emph{$C_0$-groups, commutator methods and
spectral theory of $N$-body Hamiltonians,} Progress in
Mathematics, \textbf{135}, Birkh\"auser, Basel, 1996.

\bibitem{Ba79} M. Baillet, \emph{Analyse spectrale des
op\'erateurs hermitiens d'une espace de Banach,}  J. London Math.
Soc. \textbf{19} (1979), 497--508.

 \bibitem{BaChTo02}
C. J. K. Batty, R. Chill, and Yu. Tomilov, \emph{ Strong stability
 of bounded evolution families and semigroups,} J. Funct. Anal. \textbf{193} (2002), 116--139.

\bibitem{BaWo83} H. Baumg\" artel and M.  Wollenberg, \emph{
Mathematical scattering theory,} Operator Theory: Advances and
Applications, \textbf{9}, Birkh\" auser, Basel, 1983.

\bibitem{PLM} A.~Borichev, \textit{Beurling algebras and the generalized Fourier transform}, Proceedings of the London Mathematical Society \textbf{73} (1996), 431--480.

\bibitem{BOR}
A.~Borichev, \textit{On the minimum of harmonic functions},
Journal d'Analyse Mathematique \textbf{89} (2003) 199--212.

\bibitem{BoChTo07} A.~Borichev, R.~Chill and Yu.~Tomilov, \textit{Uniqueness theorems for {\rm(}sub-{\rm)}har\-monic functions
with applications to operator theory}, Proceedings of the London Mathematical
Society \textbf{95} (2007), 687--708.

\bibitem{VOL}
J.~E.~Brennan, A.~L.~Volberg, \textit{Asymptotic values and the growth of analytic functions
 in spiral domains}, Publ.\ Mat.\ \textbf{37} (1993) 465--477.

\bibitem{BrCh90} J. Brossard and L.  Chevalier, \emph{Probleme
de Fatou ponctuel et d\'erivabilit\'e des mesures,} Acta Math.
\textbf{164} (1990),  237--263.

\bibitem{DOM}
Y.~Domar, \textit{On the existence of a largest subharmonic
minorant of a given function}, Ark.\ Mat.\ \textbf{3} (1958)
429--440.

\bibitem{Do63} W. Donoghue, \emph{ A theorem of the
Fatou type,} Monatsh. Math. \textbf{67} (1963), 225--228.

\bibitem{CuNe89} Ph. Curtis and M. Neumann, \emph{Nonanalytic
functional calculi and spectral maximal spaces,} Pacific J. Math.
\textbf{137} (1989),  65--85.

\bibitem{FoVa74}
C. Foias and F.-H. Vasilescu, \emph{Non-analytic local functional
calculus,} Czechoslovak Math. J. \textbf{24} (1974), 270--283.

\bibitem{GM}
J. B. Garnett and D. E. Marshall, \textit{Harmonic measure}. New Mathematical Monographs, 2,
Cambridge University Press, Cambridge, 2005.

\bibitem{Ge57} F. W. Gehring, \emph{The Fatou theorem and its converse,} Trans. Amer.
Math. Soc. \textbf{85} (1957), 106--121.

\bibitem{Jo82}
P. E. T. Jorgensen, \emph{Spectral theory for infinitesimal
generators of one-parameter groups of isometries: the min-max
principle and compact perturbations,} J. Math. Anal. Appl.
\textbf{90} (1982), 343--370.

\bibitem{Jo92}
P. E. T. Jorgensen, \emph{ Spectral theory for one-parameter
groups of isometries,} J. Math. Anal. Appl. \textbf{168} (1992),
131--146.

\bibitem{Kanto}
S. Kantorovitz, \emph{ Spectral theory of Banach space operators.
$C\sp{k}$-classification, abstract Volterra operators, similarity,
spectrality, local spectral analysis,} Lecture Notes in
Mathematics, \textbf{1012} Springer, Berlin, 1983.

\bibitem{Ko2}
D. Kocan, \emph{Spectral manifolds for a class of operators,}.
Illinois J. Math. \textbf{10} (1966), 605--622.

\bibitem{Ko1}
D. Kocan, \emph{ A characterization of some spectral manifolds for
a class of operators,} Illinois J. Math. \textbf{16} (1972),
359--369.

\bibitem{KOO}
P.~Koosis, \textit{The Logarithmic Integral}. I, Cambridge Studies in Advanced Mathematics, vol.~12, Cambridge University Press,
Cambridge, 1988.

\bibitem{Kr2}
B. Kritt, \emph{ A theory of unbounded generalized scalar
operators,} Proc. Amer. Math. Soc. \textbf{32} (1972), 484--490.

\bibitem{Kr1}
B. Kritt, \emph{The Fourier transform of an unbounded spectral
distribution,} Proc. Amer. Math. Soc. \textbf{35} (1972), 74--80.

\bibitem{LaNe00}
K. B. Laursen and M. Neumann, \emph{An introduction
to local spectral theory,} London Mathematical Society Monographs,
Oxford University Press, New York, 2000.

\bibitem{Lo43} L. H. Loomis, \emph{The converse of the Fatou theorem for
positive harmonic functions,} Trans. Amer. Math. Soc. \textbf{53}
(1943), 239--250.

\bibitem{Ma81} E. Marschall, \emph{Funktionalkalk\"ule f\"ur abgeschlossene lineare
Operatoren in Banachr\"aumen,} Manuscripta Math. \textbf{35}
(1981), 277--310.

\bibitem{Ma86} E. Marschall, \emph{ On the functional-calculus of
nonquasianalytic groups of operators and cosine functions,} Rend.
Circ. Mat. Palermo \textbf{35} (1986),  58--81.

\bibitem{MiMiNe} T. L.  Miller, V. G. Miller, and M. Neumann,
\emph{ Spectral subspaces of subscalar and related operators,}
Proc. Amer. Math. Soc. \textbf{132} (2004), 1483--1493.

\bibitem{Mu07}
V.  M\"uller, \emph{Spectral theory of linear operators
and spectral systems in Banach algebras,} Second edition, Operator
Theory: Advances and Applications, 139, Birkh\"auser, Basel, 2007.

\bibitem{NaRuSh82} A. Nagel, W.  Rudin, and J. Shapiro, \emph{Tangential
boundary behavior of functions in Dirichlet-type spaces,}  Ann. of
Math. (2) \textbf{116} (1982),  331--360.

\bibitem{Na89} R. Nagel, \emph{ Towards a ``matrix theory'' for unbounded
operator matrices,} Math. Z. \textbf{201} (1989),  57--68.

\bibitem{NIK}
N.~K.~Nikolski{\u\i}, \textit{Selected problems of weighted
approximation and spectral analysis}, Trudy MIAN \textbf{120}
(1974); English translation in Proc.\ of the Steklov Institute of
Math.\ \textbf{120} (1974), Amer.\ Math.\ Soc., Providence, RI
(1976), 276 pp.

\bibitem{Ru78}  W. Rudin, \emph{ Tauberian theorems for positive harmonic
functions,}  Indag. Math. \textbf{40} (1978),  376--384.

\bibitem{SHA}
V.~L.~Shapiro, \textit{The uniqueness of functions harmonic in the interior of the unit disk},
Proceedings of the London Mathematical Society \textbf{13} (1963) 639--652.


\bibitem{Va08} A. Vagharshakyan, \emph{ On the maximum principle for harmonic
functions,}  Algebra i Analiz \textbf{20} (2008), no. 3, 1--17;
translation in St. Petersburg Math. J. \textbf{20} (2009),
325--337.

\bibitem{Vr73} P. Vrbov\'a, \emph{Structure of maximal spectral spaces of
generalized scalar operators,} Czechoslovak Math. J. \textbf{23}
(1973), 493--496.

\bibitem{W}
S.~E.~Warschawski, \textit{On conformal mapping of infinite strips},
Trans.\ Amer.\ Math.\ Soc.\ \textbf{51} (1942) 280--335.

\bibitem{WOL}
F.~Wolf, \textit{The Poisson integral. A study in the uniqueness of functions}, Acta.\ Math.\ \textbf{74} (1941) 65--100.


\end{thebibliography}
\end{document}